\documentclass[10pt]{article}
\usepackage[usenames,dvipsnames]{color}
\usepackage{amsmath,amsthm,amsfonts,amssymb,multicol,amscd,amsbsy}
\usepackage{graphicx}
\usepackage{booktabs}
\usepackage{subfigure}
\usepackage{enumerate}
\usepackage{xcolor}
\usepackage{url}
\usepackage[nodayofweek]{datetime}
\usepackage[english]{babel}
\usepackage[toc,page]{appendix}
\usepackage{verbatim} %multi-line comments
\usepackage{amsthm} %theoreme & definitionen
\usepackage{enumitem}
\usepackage{enumerate}
\usepackage{bbm} %math symbols
\usepackage[normalem]{ulem} %dashed & pointed underline
\usepackage{bm}
\usepackage{hyperref}
\usepackage{amsmath}
\usepackage{mathtools}

\usepackage{extarrows}
\usepackage[normalem]{ulem}
\usepackage[margin=1in]{geometry}

\usepackage{enotez}
\let\footnote=\endnote

\usepackage{pgfplots}
\usepackage{caption}
\usepackage{authblk}

\def\1{{\mathsf 1}}

\newtheorem{thm}{Theorem}[section]
\newtheorem{lem}[thm]{Lemma}

\newtheorem{prop}[thm]{Proposition}
\newtheorem*{remark*}{Remark}

\theoremstyle{definition}

\numberwithin{equation}{section}

\newtheorem{note}[thm]{Note}

\def\dotuline{\bgroup
  \ifdim\ULdepth=\maxdimen  % Set depth based on font, if not set already
   \settodepth\ULdepth{(j}\advance\ULdepth.4pt\fi
  \markoverwith{\begingroup
  \advance\ULdepth0.08ex
  \lower\ULdepth\hbox{\kern.15em .\kern.1em}%
  \endgroup}\ULon}

\def\dashuline{\bgroup
  \ifdim\ULdepth=\maxdimen  % Set depth based on font, if not set already
   \settodepth\ULdepth{(j}\advance\ULdepth.4pt\fi
  \markoverwith{\kern.15em
  \vtop{\kern\ULdepth \hrule width .3em}%
  \kern.15em}\ULon}
\allowdisplaybreaks

\usepackage{authblk}
\begin{document}

\title{The Multiple Points of Fractional Brownian Motion}
\author{Mark Landry\footnote{Michigan State University. Email:\href{mailto: landrym5@msu.edu}{ landrym5@msu.edu}}, Cheuk Yin Lee\footnote{Michigan State University. Email:\href{mailto: leecheu15@msu.edu}{ leecheu1@msu.edu}}, Paige Pearcy\footnote{University of Arizona. Email:\href{mailto: paigepearcy@email.arizona.edu}{ paigepearcy@email.arizona.edu}}}

\setlength{\columnsep}{5pt}
\date{}
\maketitle

\begin{abstract}
Nils Tongring (1987) proved sufficient conditions for a compact set to contain $k$-tuple points of a Brownian motion.
In this paper, we extend these findings to the fractional Brownian motion.
Using the property of strong local nondeterminism, we show that if $B$ is a fractional Brownian 
motion in $\mathbb{R}^d$ with Hurst index $H$ such that $Hd=1$, and $E$ is a fixed, 
nonempty compact set in $\mathbb{R}^d$ with positive capacity with respect to the  function $\phi(s) = (\log_+(1/s))^k$, then $E$ contains $k$-tuple points with 
positive probability. 
For the $Hd > 1$ case, the same result holds with the function replaced by $\phi(s) = s^{-k(d-1/H)}$.
\end{abstract}

\section{Introduction}

It is a well known result that almost all sample paths of Brownian motion in $\mathbb{R}^2$ have $k$-tuple points for any positive integer $k$ \cite{DEK}. It is also well known that a Brownian path in $\mathbb{R}^2$ will hit a compact set if and only if the set has positive logarithmic capacity \cite{K44}. Nils Tongring \cite{Tongring} combined these results in 1987 to prove that if $B$ is a Brownian motion in $\mathbb{R}^2$, and $E$ is a fixed compact set in the plane with positive capacity with respect to the function $\phi(s)=(\log(1/s))^k$, $k$ a positive integer, then $E$ contains $k$-tuple points for almost all paths.
We say that $x$ is a \emph{$k$-tuple point} (or \emph{$k$-multiple point}) for the path $\omega$ if there are times $t_1 < t_2 < \cdots < t_k$ such that $x = B(t_1, \omega) = B(t_2, \omega) = \cdots = B(t_k, \omega)$. 

In this paper, we extend Tongring's results to the case of fractional Brownian motion, which is known as an extension of Brownian motion with the dropped required property of independent increments of Brownian motion. The fractional Brownian motion of Hurst index $0 < H < 1$ is defined as a mean zero Gaussian process $\{B_t : t \ge 0\}$ with continuous sample paths and covariance function
$\mathbb{E}(B_t B_s)=(|t|^{2H}+|s|^{2H}-|t-s|^{2H})/2$.
The fractional Brownian motion in $\mathbb{R}^d$ is defined as $B_t = (B_t^1, \dots, B^d_t)$ where $B^1, \dots, B^d$ are i.i.d.~copies of one-dimensional fractional Brownian motion.
The fractional Brownian motion has the scaling property that for any $c > 0$, the process $\{ c^{-H} B_{ct} : t \ge 0 \}$ is also a fractional Brownian motion.

The existence of multiple points for fractional Brownian motion was studied by several authors. The results of K\^{o}no \cite{K78}, Goldman \cite{G81} and Rosen \cite{R84} show that if $k > (k-1)Hd$, then $k$-tuple points exist in any interval $T \subset (0, \infty)$; and if $k < (k-1)Hd$, then $k$-tuple points do not exist.
Talagrand \cite{T98} proved that in the critical case $k = (k-1)Hd$, $k$-tuple points do not exist. In fact, their results cover the multiparameter case $(t \in \mathbb{R}^N)$.
Rosen \cite{R84} also studied the Hausdorff dimension of multiple times using local times.

We study sufficient conditions for a fixed set $E$ in $\mathbb{R}^d$ to contain multiple points of a fractional Brownian motion in $\mathbb{R}^d$.
Our main results are Theorem \ref{ktuplepointsHd=1} and \ref{ktupleHd>1}, concerning the cases $Hd = 1$ and $Hd > 1$ respectively.
We show that if $Hd = 1$ (resp.~$Hd>1$) and $E$ is a fixed, nonempty compact set in 
$\mathbb{R}^d$ with positive capacity with respect to the function 
$\phi(s) = (\log_+(1/s))^k$ (resp.~$\phi(s) = s^{-k(d-1/H)}$), then $E$ contains 
$k$-tuple points with positive probability. To avoid negative values of the logarithm, we define $\log_+(x) = \max\{\log(x), 0\}$.
Our results extend Theorem 1 and 3 of Tongring \cite{Tongring}.

For the case that $Hd<1$, it is known that for any fixed point $x$ in $\mathbb{R}^d$, the Hausdorff dimension of the level set $B^{-1}(x) = \{ t \ge 0 : B(t) = x \}$ of the fractional Brownian motion is a.s.~equal to $1 - Hd$, which is positive (see \cite{A, K}).
It follows that that $B^{-1}(x)$ contains uncountably many time points and therefore, $x$ is a multiple point of infinitely large multiplicity.

Unlike the Brownian motion, the fractional Brownian motion (for $H \ne 1/2$) does not have independent increments. 
In order to estimate the joint distributions, our proofs rely on the property of strong local nondeterminism for the fractional Brownian motion due to Pitt \cite{Pitt} (also see Proposition \ref{Prop:LND} below).
The concept of local nondeterminism (LND) was first introduced by Berman \cite{B73} to study local times of Gaussian processes. 
We refer to \cite{B87, C78, Pitt} for general conditions of LND for Gaussian processes.
LND is useful for studying small ball probability, H\"{o}lder conditions of local times, Hausdorff measure of ranges and level sets, and other properties of Gaussian random fields (see e.g.~\cite{X97, X08}).

Let us explain notations and gather useful formulas before we end the introduction.
For any nonempty Borel set $E$ in $\mathbb{R}^d$, the \emph{capacity} of $E$ respect to a function $\phi: [0, \infty) \to [0, \infty]$, or the \emph{$\phi$-capacity}, can be defined as the quantity
\[ C_\phi(E) = \bigg[\inf_\mu \int_E \int_E \phi(|z_1 - z_2|) \mu(dz_1) \mu(dz_2)\bigg]^{-1} \]
where the infimum is taken over all probability measures $\mu$ on $E$.
It follows that $E$ has positive capacity with respect to $\phi$ if and only if there exists a probability measure $\mu$ on $E$ such that
\[ \int_E \int_E \phi(|z_1 - z_2|) \mu(dz_1) \mu(dz_2) < \infty.\]

Recall that for any mean zero Gaussian vector $(Z, Z_1, \dots, Z_n)$, the conditional variance of $Z$ given $Z_1, \dots, Z_n$ is equal to the squared $L^2(\mathbb{P})$-distance of $Z$ from the subspace generated by $Z_1, \dots, Z_n$, namely,
\begin{equation}\label{cond_var=L2_dist}
\mathrm{Var}(Z|Z_1, \dots, Z_n) = \inf_{a_1, \dots, a_n \in \mathbb{R}} \mathbb{E}\bigg[\Big(Z - \sum_{j=1}^n a_j Z_j\Big)^2 \bigg].
\end{equation}
The covariance matrix of the Gaussian vector $(Z_1, \dots, Z_n)$ is denoted by 
$\mathrm{Cov}(Z_1, \dots, Z_n)$.
A useful formula for the determinant of the covariance matrix is:
\begin{equation}\label{Eq:detcov}
\det\mathrm{Cov}(Z_1, \dots, Z_n) = \mathrm{Var}(Z_1) \prod_{j=2}^n \mathrm{Var}(Z_j|Z_1, \dots, Z_{j-1}).
\end{equation}
An oft-used result in this paper will be the following proposition.

\begin{prop}\label{Prop:LND}
Let $B$ be a fractional Brownian motion of Hurst index $H$. There exists a constant $0<C_0 \le1$ such that for all $n \ge 1$, for all $t, s_1, ..., s_n \ge 0$, 
\begin{equation*}
C_0 \min\limits_{1\le i \le n}|t-s_i|^{2H}\le \mathrm{Var}(B_t|B_{s_1},\cdots,B_{s_n})\le \min\limits_{1\le i \le n}|t-s_i|^{2H}.
\end{equation*}
\end{prop}

The upper bound follows easily from \eqref{cond_var=L2_dist} and the fact that $\mathbb{E}[(X_t - X_s)^2] = |t-s|^{2H}$.
The lower bound was proved by Pitt \cite[Lemma 7.1]{Pitt} and is known as the property of strong local nondeterminism. 
Formula \eqref{Eq:detcov} and Proposition \ref{Prop:LND} together allow us to estimate the joint distributions of fractional Brownian motion.

Throughout this paper, $C$ denotes a positive finite constant whose value may vary in each appearance, and $C_1, C_2, \dots$ denote specific constants.
For any $z \in \mathbb{R}^d$, $|z|$ denotes its Euclidean norm.
Also, $D(z, \varepsilon)$ denotes the closed ball in $\mathbb{R}^d$ centered at $z$ of radius $\varepsilon$ under the Euclidean norm, i.e.~$D(z, \varepsilon) = \{ x \in \mathbb{R}^d : |x-z| \le \varepsilon \}$.

\section{Multiple Points of Fractional Brownian Motion with $Hd = 1$}

\begin{lem}\label{cap_scaling}
Let $E$ be a Borel set in $\mathbb{R}^d$ and $\phi(s) = (\log_+(1/s))^k$, where $k$ is a positive integer.
If $E$ has positive $\phi$-capacity, then $\lambda E$ has positive $\phi$-capacity for any $\lambda > 0$.
\end{lem}

\begin{proof}
Suppose that $E$ has positive $\phi$-capacity and $\mu$ is a probability measure on $E$ such that
\begin{equation}\label{finite_energy}
\int_E \int_E \bigg(\log_+\frac{1}{|z_1-z_2|}\bigg)^k \mu(dz_1) \mu(dz_2) < \infty.
\end{equation}
Define the probability measure $\mu_\lambda$ on $\lambda E$ by $\mu_\lambda(A) = \mu(\lambda^{-1}A)$ for any Borel subset $A$ of $\lambda E$. We claim that
\begin{equation}\label{finite_energy_scaled}
\int_{\lambda E} \int_{\lambda E} \bigg(\log_+\frac{1}{|z_1-z_2|}\bigg)^k \mu_\lambda(dz_1) \mu_\lambda(dz_2) < \infty. 
\end{equation}
By changing variables, 
\[ \int_{\lambda E} \int_{\lambda E} \bigg(\log_+\frac{1}{|z_1-z_2|}\bigg)^k \mu_\lambda(dz_1) \mu_\lambda(dz_2) = \int_E \int_E \bigg(\log_+\frac{1}{\lambda|z_1-z_2|}\bigg)^k \mu(dz_1) \mu(dz_2). \]
If $\lambda \ge 1$, then $\log_+(\frac{1}{\lambda|z_1-z_2|}) \le \log_+(\frac{1}{|z_1-z_2|})$ and \eqref{finite_energy_scaled} clearly holds. Suppose $\lambda < 1$. Then
\begin{align*}
\int_E \int_E \bigg(\log_+\frac{1}{\lambda|z_1-z_2|}\bigg)^k \mu(dz_1) \mu(dz_2)
&= \iint_{\substack{z_1, z_2 \in E, \\|z_1 - z_2| \le \lambda^{-1}}} \bigg(\log \frac{1}{\lambda|z_1-z_2|}\bigg)^k \mu(dz_1) \mu(dz_2)\\
& = \iint_{\substack{z_1, z_2 \in E, \\|z_1 - z_2| \le \lambda^{-1}}} \bigg(\log \frac{1}{\lambda} + \log \frac{1}{|z_1-z_2|}\bigg)^k \mu(dz_1) \mu(dz_2).
\end{align*}
We can split the last integral into two parts: 
\begin{align*}
\Bigg(\iint_{\substack{z_1, z_2 \in E, \\|z_1 - z_2| \le 1}} + \iint_{\substack{z_1, z_2 \in E, \quad\\1 < |z_1 - z_2| \le \lambda^{-1}}} \Bigg) \bigg(\log \frac{1}{\lambda} + \log \frac{1}{|z_1-z_2|}\bigg)^k \mu(dz_1) \mu(dz_2).
\end{align*}
For the first part, we can use the inequality $(a+b)^k \le 2^{k-1}(a^k+b^k)$ for $a, b \ge 0$ and $k \ge 1$. For the second part, we note that $0 \le \log(1/\lambda) - \log |z_1-z_2| \le \log(1/\lambda)$. It follows that
\begin{align*}
\int_E \int_E \bigg(\log_+\frac{1}{\lambda|z_1-z_2|}\bigg)^k \mu(dz_1) \mu(dz_2) 
&\le 2^{k-1} \iint_{\substack{z_1, z_2 \in E, \\|z_1 - z_2| \le 1}} \Bigg[\left(\log\frac{1}{\lambda}\right)^k + \bigg( \log\frac{1}{|z_1-z_2|}\bigg)^k \Bigg]\mu(dz_1) \mu(dz_2) \\
&\quad + \iint_{\substack{z_1, z_2 \in E,\quad\\1 < |z_1 - z_2| \le \lambda^{-1}}} \left(\log \frac{1}{\lambda}\right)^k \mu(dz_1) \mu(dz_2).
\end{align*}
The right-hand side is finite by \eqref{finite_energy}. Hence \eqref{finite_energy_scaled} holds and $\lambda E$ has positive $\phi$-capacity.
\end{proof}

The following is our main result for the case of $Hd=1$.

\begin{thm}
\label{ktuplepointsHd=1}
If B is a fractional Brownian motion in $\mathbb{R}^d$ with Hurst index $H$ such that $Hd=1$, and $E$ is a fixed, nonempty compact set in $\mathbb{R}^d$ with positive capacity with respect to the function $\phi(s) = (\log_+(1/s))^k$, then E contains $k$-tuple points with positive probability.
\end{thm}

\begin{proof}
Suppose that $E$ has positive capacity with respect to the function $\phi(s) = (\log_+(1/s))^k$. That is, there is a probability measure $\mu$ defined on $E$ such that
\begin{equation}\label{assump}
\int_{E}\int_{E}\bigg(\log_+\frac{1}{|z_1-z_2|}\bigg)^k\mu(dz_1)\mu(dz_2) < \infty.
\end{equation}
By Lemma \ref{cap_scaling} and the scaling property of the fractional Brownian motion, we can assume without loss of generality that the set $E$ is contained in a ball of radius $1/3$ centered at the origin, so that $|z_1 - z_2| \le 2/3$ and $\log_+(1/|z_1 - z_2|)$ has a positive lower bound: 
\begin{equation}\label{Eq:log_lower_bd}
\log_+(1/|z_1 - z_2|) \ge \log(3/2) > 0, \quad \forall\, z_1, z_2 \in E.
\end{equation}

Let $u \in \mathbb{R}^+$. Let $\mathbf{K}_\varepsilon(u)$ be an indicator function such that
\[ 
\mathbf{K}_\varepsilon(u) = \begin{cases} 
1, & u\leq \varepsilon, \\
0, & u> \varepsilon.
\end{cases}
\]
For $0 < \varepsilon < 1$, consider the following integral:
\begin{align}\label{Int:I_epsilon}
I_{\varepsilon}= \frac{1}{\varepsilon^{kd}}\int_{E}\mu(dz) \prod_{j=1}^k \int_{2j-1}^{2j}ds_j \textbf{K}_\varepsilon({|B_{s_j}-z|}).
\end{align}
To prove positive probability of $k$-tuple points in the set $E$, we will show that $\mathbb{P}(I_\varepsilon>0)$ has a positive lower bound independent of $\varepsilon$. 
Since $I_\varepsilon$ is non-negative, $\mathbb{E}(I_\varepsilon) = \mathbb{E}(I_\varepsilon \cdot {\bf 1}_{\{ I_\varepsilon > 0 \}})$.
By the Cauchy--Schwarz inequality, we have
\begin{equation}\label{Eq:CS}
\mathbb{P}(I_\varepsilon>0)\ge\frac{\mathbb{E}(I_\varepsilon)^2}{\mathbb{E}(I_\varepsilon^2)}.
\end{equation}
To prove that right-hand side is positive, we show that the numerator is bounded below, and the denominator is bounded above; both bounds are also independent of $\varepsilon$.

First we prove a lower bound for $\mathbb{E}(I_\varepsilon)^2$.
By Fubini's theorem,
\begin{equation}\label{Eq:E(I_ep)}
\mathbb{E}(I_\varepsilon)= \frac{1}{\varepsilon^{kd}}\int_{E}\mu(dz) \int_1^2 ds_1 \int_3^4 ds_2 \cdots \int_{2k-1}^{2k} ds_k\, \mathbb{P}\bigg( \bigcap_{j=1}^k \big\{|B_{s_j}-z|\le\varepsilon\big\}\bigg).
\end{equation}
Let $s_j \in [2j-1, 2j]$ and write $B_{s_j} = (B_{s_j}^1, \dots, B_{s_j}^d)$ for $j = 1, \dots, k$.
Let $\Sigma_0 = \mathrm{Cov}(B_{s_1}^1, \dots, B_{s_k}^1)$ be the covariance matrix of the $k$-dimensional Gaussian vector $(B_{s_1}^1, \dots, B_{s_k}^1)$, and let $\Sigma = \mathrm{Cov}(B_{s_1}, \dots, B_{s_k})$ be the covariance matrix of the $kd$-dimensional Gaussian vector $(B_{s_1}, \dots, B_{s_k})$.
Then,
\begin{align*}
\mathbb{P}\bigg( \bigcap_{j=1}^k \big\{|B_{s_k} - z| \le \varepsilon \big\}\bigg) = \int\cdots\int_{D(z, \varepsilon)^k} \frac{1}{(2\pi)^{kd/2} (\det\Sigma)^{1/2}} 
\exp\left(-\frac{x \Sigma^{-1}x^T}{2} \right) dx_1 \cdots dx_k,
\end{align*} 
where $x$ denotes the row vector $(x_1, \dots, x_k)$, with $x_j \in \mathbb{R}^d$, and $x^T$ denotes its transpose.
Since $B^1, \dots, B^d$ are i.i.d., $\det\Sigma = (\det\Sigma_0)^d$, and by \eqref{Eq:detcov}, we have 
\[\det\Sigma_0 = \mathrm{Var}(B_{s_1}^1) \prod_{j=2}^k \mathrm{Var}(B_{s_j}^1|B_{s_1}^1, \dots, B_{s_{j-1}}^1).\]
Then by Proposition \ref{Prop:LND}, we get that
\[\det\Sigma \le [s_1^{2H} (s_2-s_1)^{2H} \cdots (s_k-s_{k-1})^{2H}]^d \le [2^{2H} 3^{2H(k-1)}]^d=: C_1. \]
Let $\lambda_1 \le \dots \le \lambda_k$ be the eigenvalues of the matrix $\Sigma_0$.
Then,
\[ x \Sigma^{-1} x^T \le \frac{1}{\lambda_1}(|x_1|^2 + \dots + |x_k|^2).  \]
Recall that the largest eigenvalue of a real symmetric $k\times k$ matrix $A$ is equal to $\sup\{ v A v^T : v \in \mathbb{R}^k, |v| = 1 \}$.
Since the entries of $\Sigma_0$ are bounded as a function of $(s_1, \dots, s_k)$ on the compact set $\prod_{j=1}^k[2j-1, 2j]$, the eigenvalues of $\Sigma_0$ are bounded by a constant $C_2$.
It follows that
\[ \det\Sigma_0 = \prod_{j=1}^k \lambda_j \le \lambda_1 C_2^{k-1}. \]
Also, by Proposition \ref{Prop:LND}, we have
\begin{align*}
\det\Sigma_0 &=\mathrm{Var}(B_{s_1}^1) \prod_{j=2}^k \mathrm{Var}(B_{s_j}^1|B_{s_1}^1, \dots, B_{s_{j-1}}^1)\\
& \ge s_1^{2H} \prod_{j=2}^k[C_0(s_j-s_{j-1})^{2H}] \ge C_0^{k-1}.
\end{align*}
Hence $\lambda_1\ge (C_0/C_2)^{k-1}$. Note that for $x_j \in D(z, \varepsilon)$, $|x_j|\le |z| + \varepsilon < 1$, and $|x_1|^2 + \dots + |x_k|^2 < k$, so
\begin{align*}
\mathbb{P}\bigg(\bigcap_{j=1}^k\big \{|B_{s_j}-z|\le \varepsilon\big\}\bigg) 
&\ge \frac{1}{C_1^{1/2}(2\pi)^{kd/2}} \int\dots \int_{D(z, \varepsilon)^k} \exp\bigg(-\frac{k}{2(C_0/C_2)^{k-1}}\bigg)dx_1\cdots dx_k
 = C \varepsilon^{kd}.
\end{align*}
Plugging this back into \eqref{Eq:E(I_ep)} gives $\mathbb{E}(I_\varepsilon)^2 \ge C_3$, where $C_3 > 0$ is a constant independent of $\varepsilon$.

Next we show that $\mathbb{E}(I_\varepsilon^2)$ has a constant upper bound independent of $\varepsilon$.
Following \cite{Tongring}, we divide $\mathbb{E}(I_\varepsilon^2)$ into two parts, $F$ and $S$. The first part $F$ is given by restricting the variables $z_1, z_2$ to values for which $|z_1 - z_2| \leq 4\varepsilon$, and the second part $S$ is given by the remaining  $z_1, z_2$ for which $|z_1 - z_2| > 4\varepsilon$. That is,
\[ \mathbb{E}(I_\varepsilon^2) = F + S,\]
where $F$ and $S$ are defined by
\begin{multline*}
F= \frac{1}{\varepsilon^{2kd}}\iint_{\substack{z_1, z_2 \in E,\\ |z_1 - z_2| \le 4\varepsilon}} d\mu (z_1) d\mu (z_2) \iint_{[1,2]^2} ds_1 d\widehat{s}_1 \iint_{[3, 4]^2} ds_2 d\widehat{s}_2\\\cdots
\iint_{[2k-1, 2k]^2}ds_k d\widehat{s}_k \,\mathbb{P}\bigg(\bigcap_{j=1}^k\big\{|B_{s_j}-z_1|\le\varepsilon,|B_{\widehat{s}_j}-z_2|\le\varepsilon\big\}\bigg)
\end{multline*}
and
\begin{multline*}
S= \frac{1}{\varepsilon^{2kd}}\iint_{\substack{z_1, z_2 \in E,\\ |z_1-z_2| > 4\varepsilon}} d\mu (z_1)d\mu (z_2) \iint_{[1,2]^2} ds_1 d\widehat{s}_1 \iint_{[3, 4]^2} ds_2 d\widehat{s}_2\\\cdots
\iint_{[2k-1, 2k]^2}ds_k d\widehat{s}_k \,\mathbb{P}\bigg(\bigcap_{j=1}^k\big\{|B_{s_j}-z_1|\le\varepsilon,|B_{\widehat{s}_j}-z_2|\le\varepsilon\big\}\bigg).
\end{multline*}
    
Let us consider $S$. Assume $|z_1 - z_2| > 4\varepsilon$. First, we fix $s_j, \widehat s_j \in [2j-1, 2j]$ for $j =1, \dots, k$ and consider the joint probability that appears in $S$.
We may assume without loss of generality that $s_1<\widehat{s}_1 < \cdots <s_k<\widehat{s}_k$, because for other possible orderings, we can interchange $s_j$ and $\widehat s_j$, and follow the same argument. 
By the triangle inequality, the joint probability is
\begin{align*}
\mathbb{P}\bigg(\bigcap_{j=1}^k\big\{|B_{s_j}-z_1|\le\varepsilon,|B_{\widehat{s}_j}-z_2|\le\varepsilon\big\}\bigg)
 \le \mathbb{P}\bigg( \bigcap_{j=1}^k \big\{ |B_{s_j} - z_1| \le \varepsilon, |B_{\widehat{s}_j} - B_{s_j} - (z_2-z_1)| \le 2\varepsilon \big\} \bigg).
\end{align*}
Now, we normalize the random variables $B_{s_1}, B_{\widehat{s}_1} -  B_{s_1}, \dots, B_{s_k}, B_{\widehat{s}_k} -  B_{s_k}$ by dividing by their respective standard deviations. 
We will see that it allows us to obtain a lower bound involving their inverse covariance matrix.
Let
\[X_j=\frac{B_{s_j}}{s_j^H}, \quad \widehat{X}_j=\frac{B_{\widehat{s}_j}-B_{s_j}}{|\widehat{s}_j-s_j|^H} \]
for $ = 1, \dots, k$.
Define the vector $X_n = (X_n^1,X_n^2, ..., X_n^d)$. 
To simplify notations, let us denote 
\[y_j=\frac{z_1}{s_j^H}, \quad 
\widehat{y}_j=\frac{z_2-z_1}{|\widehat{s}_j-s_j|^H}\]
and
\[
\varepsilon_j = \frac{\varepsilon}{s_j^H}, \quad 
\widehat{\varepsilon}_j = \frac{\varepsilon}{|\widehat{s}_j-s_j|^H}
\]
for $j = 1, \dots, k$.
Then
\begin{align}\label{Eq:UBint}
\notag &\quad\, \mathbb{P}\bigg( \bigcap_{j=1}^k \big\{ |B_{s_j} - z_1| \le \varepsilon, |B_{\widehat{s}_j} - B_{s_j} - (z_2-z_1)| \le 2\varepsilon \big\} \bigg)\\
\notag&= \mathbb{P}\bigg( \bigcap_{j=1}^k \big\{ |X_j - y_j| \le \varepsilon_j, |\widehat{X}_j - \widehat{y}_j| \le 2\widehat{\varepsilon}_j \big\} \bigg)\\
& = \int\dots\int_D \frac{1}{(2\pi)^{kd}(\mathrm{det}\Sigma)^{1/2}}\;\exp\bigg({-\frac{x\Sigma^{-1}x^T}{2}}\bigg)dx_1d\widehat{x}_1\cdots dx_kd\widehat{x}_k
\end{align}
where $D = \prod_{j=1}^k [D(y_j, \varepsilon_j) \times D(\widehat{y}_j, 2\widehat{\varepsilon}_j)]$, $x = (x_1, \widehat{x}_1, \dots, x_k, \widehat{x}_k)$ and
\[ \Sigma = \mathrm{Cov}(X_1, \widehat{X}_1, \dots, X_k, \widehat{X}_k). \]
By \eqref{Eq:detcov}, $\mathrm{det}\Sigma$ is equal to
\[ \Big\{\mathrm{Var}(X_1^1)\mathrm{Var}(\widehat{X}_1^1|X_1^1)\prod_{j=2}^k\big[\mathrm{Var}(X_j^1|X_1^1, \widehat{X}_1^1, \dots, X_{j-1}^1, \widehat{X}_{j-1}^1) \mathrm{Var}(\widehat{X}_j^1|X_1^1, \widehat{X}_1^1, \dots, X_{j-1}^1, \widehat{X}_{j-1}^1, X_j^1)\big]\Big\}^d.\]
Clearly, $\mathrm{Var}(X^1_1) = 1$. 
To find lower bounds for the conditional variances, we use the property of local nondeterminism.
Using \eqref{cond_var=L2_dist} and Proposition \ref{Prop:LND}, we get that
\begin{align*}
\mathrm{Var}(\widehat{X}^1_1|X^1_1)
\ge \frac{1}{|\widehat{s}_1 - s_1|^{2H}}\mathrm{Var}(B^1_{\widehat{s}_1}|B^1_{s_1}) \ge  C_0.
\end{align*}
Similarly, for all $j = 2, \dots, k$, we have
\begin{align*}
\mathrm{Var}(X_j^1|X_1^1,\widehat{X}_1^1, \dots, X_{j-1}^1, \widehat{X}_{j-1}^1) 
& \ge \frac{1}{s_j^{2H}} \mathrm{Var}(B^1_{s_j}|B^1_{s_1}, B^1_{\widehat{s}_1}, \dots, B^1_{s_{j-1}}, B^1_{\widehat{s}_{j-1}})\\
& \ge \frac{C_0 |s_j - \widehat{s}_{j-1}|^{2H}}{s_j^{2H}} \ge \frac{C_0}{(2j)^{2H}}
\end{align*}
and
\begin{align*}
\mathrm{Var}(\widehat{X}_j^1|X_1^1,\widehat{X}_1^1, \dots, X_{j-1}^1, \widehat{X}_{j-1}^1, X_j^1) 
& \ge \frac{1}{|\widehat{s}_j-s_j|^{2H}} \mathrm{Var}(B^1_{\widehat{s}_j}|B^1_{s_1}, B^1_{\widehat{s}_1}, \dots, B^1_{s_{j-1}}, B^1_{\widehat{s}_{j-1}}, B^1_{s_j}) \ge C_0.
\end{align*}
It follows that 
\begin{equation}\label{Eq:detSigma}
\det\Sigma \ge C
\end{equation}
for some constant $C > 0$ independent of $\varepsilon$.

Let us consider the exponential function of the joint density in \eqref{Eq:UBint}.
Let $\lambda_{\mathrm{max}}$ denote the maximum eigenvalue of $\Sigma$.
Then for any $x = (x_1, \widehat{x}_1, \dots, x_k, \widehat x_k) \in D$,
\begin{align*}
x\Sigma^{-1}x^T 
& \ge \frac{1}{\lambda_{\mathrm{max}}}|x|^2
 \ge \frac{1}{\lambda_{\mathrm{max}}}(|\widehat x_1|^2 + \cdots + |\widehat x_k|^2).
\end{align*}
By normalizing the random variables, we are able to use Gershgorin's circle theorem \cite[Theorem 6.1.1]{HJ} to find this bound. Since the entries in the $\Sigma$ matrix are the correlation coefficients between the random variables, the diagonal entries are 1 and the sum of the absolute values of the non-diagonal entries in each row is bounded by $2k-1$, thus $\lambda_{\mathrm{max}} \le 1 + (2k-1) = 2k$. It follows that
\begin{equation}\label{Eq:UBexp}
 \exp\Big(-\frac{x\Sigma^{-1}x^T}{2}\Big)
\le \exp\bigg(-\frac{|\widehat x_1|^2+\dots + |\widehat x_k|^2}{4k}\bigg).
\end{equation}
Using \eqref{Eq:detSigma} and \eqref{Eq:UBexp}, with a change of variables, we get that the integral in \eqref{Eq:UBint} is bounded above by
\begin{align*}
C \prod_{j=1}^k \iint_{D(0, \varepsilon_j) \times D(0, 2\widehat{\varepsilon}_j)} \exp\bigg(-\frac{|\widehat x_j+\widehat y_j|^2}{4k}\bigg)dx_j d\widehat x_j.
\end{align*}
Note that the assumption $|z_1-z_2| > 4\varepsilon$ implies $2\varepsilon |\widehat{s}_j - s_j|^{-H} \le |\widehat{y}_j|/2$ for all $j$.
By the triangle inequality, if $\widehat x_j \in D(0, 2\widehat{\varepsilon}_j)$, then
\begin{align*}
|\widehat{y}_j + \widehat{x}_j| \ge |\widehat{y}_j| - \frac{2\varepsilon}{|\widehat{s}_j - s_j|^H} \ge |\widehat{y}_j| - \frac{|\widehat{y}_j|}{2} = \frac{|\widehat{y}_j|}{2},
\end{align*}
which implies
\[ \exp\bigg(-\frac{|\widehat{x}_j+\widehat{y}_j|^2}{4k}\bigg) \le \exp\bigg(-\frac{|\widehat{y}_j|^2}{16k}\bigg).\]
Recall that $\widehat\varepsilon_j = \varepsilon/|\widehat{s}_j - s_j|^{H}$ and $Hd = 1$.
It follows that
\begin{align*}
 \mathbb{P}\bigg( \bigcap_{j=1}^k \big\{ |B_{s_j} - z_1| \le \varepsilon, |B_{\widehat{s}_j} - B_{s_j} - (z_2-z_1)| \le 2\varepsilon \big\} \bigg)
& \le C \prod_{j=1}^k \iint_{D(0, \varepsilon_j) \times D(0, 2\widehat \varepsilon_j)}\exp\bigg( - \frac{|\widehat{y}_j|^2}{16k} \bigg) dx_j d\widehat x_j\\
& \le C' \varepsilon^{2kd} \prod_{j=1}^k\left[\frac{1}{|\widehat{s}_j - s_j|} \exp\Big( - \frac{|\widehat{y}_j|^2}{16k} \Big)\right].
\end{align*}

Returning to our integral S, we have
\begin{align}\label{integralSwithCOV}
S \le C\iint_{\substack{z_1, z_2 \in E,\\|z_1-z_2|>4\varepsilon}}\mu(dz_1)\mu(dz_2) \prod_{j=1}^k \iint_{[2j-1, 2j]^2} \frac{1}{|\widehat{s}_j-s_j|}\exp\bigg(\frac{-|\widehat{y}_j|^2}{16k}\bigg) ds_j d\widehat{s}_j.
\end{align}
Let us write
\begin{align*}
\iint_{[2j-1, 2j]^2}\frac{1}{|\widehat{s}_j-s_j|}\exp\bigg(-\frac{|\widehat{y_j}|^2}{16k}\bigg) ds_jd\widehat{s}_j = L_j+M_j,
\end{align*}
where
\begin{align*}
L_j = \iint_{\substack{s_j, \widehat{s}_j \in [2j-1,2j],\quad\\|\widehat{s}_j-s_j|<|z_1-z_2|^{1/H}}}\frac{1}{|\widehat{s}_j-s_j|}\exp\bigg(-\frac{|\widehat{y}_j|^2}{16k}\bigg)ds_jd\widehat{s}_j
\end{align*}
and
\begin{align*}
M_j = \iint_{\substack{s_j, \widehat{s}_j \in [2j-1,2j],\quad\\|\widehat{s}_j-s_j|\ge|z_1-z_2|^{1/H}}}\frac{1}{|\widehat{s}_j-s_j|}\exp\bigg(-\frac{|\widehat{y}_j|^2}{16k}\bigg)ds_jd\widehat{s}_j.
\end{align*}
For the integral $L_j$, recalling the definition of $\widehat{y}_j$ and noting that the function $y \mapsto |y|^{1/H} \exp(-|y|^2/16k)$ is bounded on $\mathbb{R}^d$, we see that
\begin{align*}
L_j &=\iint_{\substack{s_j, \widehat{s}_j \in [2j-1,2j],\quad\\|\widehat{s}_j-s_j|<|z_1-z_2|^{1/H}}}\frac{1}{|z_1-z_2|^{1/H}}|\widehat{y}_j|^{1/H}\exp\bigg(\frac{-|\widehat{y}_j|^2}{16k}\bigg) ds_jd\widehat{s}_j\\
& \le \iint_{\substack{s_j, \widehat{s}_j \in [2j-1,2j],\quad\\|\widehat{s}_j-s_j|<|z_1-z_2|^{1/H}}}\frac{C}{|z_1-z_2|^{1/H}}ds_jd\widehat{s}_j\\
& \le C'
\end{align*}
and by \eqref{Eq:log_lower_bd}, we have
\[ L_j \le C \log\frac{1}{|z_1-z_2|} . \]
Considering $M_j$, since $\exp(-|y|^2/16k) \le 1$ for all $y \in \mathbb{R}^d$, we get that
\begin{align*}
M_j \le \iint_{\substack{s_j, \widehat{s}_j \in [2j-1,2j],\quad\\ |\widehat{s}_j-s_j|\ge|z_1-z_2|^{1/H}}}\frac{1}{|\widehat{s}_j-s_j|}ds_jd\widehat{s}_j.
\end{align*}
Elementary calculations show that for any $a \in \mathbb{R}$ and $0 < x < 1$,
\begin{equation}\label{int_est_log}
\iint_{\substack{s, \widehat{s} \in [a ,a+1],\\ |\widehat{s}-s| > x}} \frac{1}{|\widehat{s}-s|} ds\, d\widehat{s} = 2\log\frac{1}{x} - 2(1-x). 
\end{equation}
It follows that 
\[ M_j \le \frac{2}{H}\log\frac{1}{|z_1-z_2|}. \]
Then we can combine the estimates for $L_j$ and $M_j$ to deduce from \eqref{integralSwithCOV} that $S$ has following upper bound:
\begin{align}\label{bound_S}
S \le C \iint_{\substack{z_1, z_2 \in E,\\|z_1-z_2|>4\varepsilon}}
\bigg(\log\frac{1}{|z_1-z_2|}\bigg)^k\mu(dz_1)\mu(dz_2).
\end{align}

Let us consider $F$. Assume $|z_1 - z_2| \leq 4\varepsilon$. Again, we first consider the probability 
\[ \mathbb{P}\bigg(\bigcap_{j=1}^k\big\{|B_{s_j}-z_1|\le\varepsilon,|B_{\widehat{s}_j}-z_2|\le\varepsilon\big\}\bigg)\]
with normalized terms. We use the same process from \eqref{Eq:UBint} to \eqref{Eq:UBexp}, and integrate $dx_1, \dots, dx_k$ to find
\begin{align*}
 \mathbb{P}\bigg(\bigcap_{j=1}^k\big\{|B_{s_j}-z_1|\le\varepsilon,|B_{\widehat{s}_j}-z_2|\le\varepsilon\big\}\bigg)
& \le \int\cdots\int_D \frac{1}{(2\pi)^{kd}(\det\Sigma)^{1/2}} \exp\bigg(-\frac{x\Sigma^{-1}x^T}{2} \bigg) dx_1 d\widehat x_1 \cdots dx_k d\widehat x_k\\
& \le C \varepsilon^{kd} \prod_{j=1}^k \int_{D(\widehat y_j, 2\widehat \varepsilon_j)} \exp\bigg( -\frac{|\widehat x_j|^2}{16k}\bigg) d\widehat x_j.
\end{align*}
It follows that
\begin{align}\label{Eq:bound_F}
\begin{aligned}
F \le C\varepsilon^{-kd}\iint_{\substack{z_1, z_2 \in E,\\ |z_1 - z_2| \le 4\varepsilon}} 
\mu(dz_1) \mu(dz_2) \prod_{j=1}^k \iint_{[2j-1, 2j]^2} ds_j d\widehat s_j \int_{D(\widehat y_j, 2\widehat \varepsilon_j)} \exp\bigg( -\frac{|\widehat x_j|^2}{16k}\bigg) d\widehat x_j.
\end{aligned}
\end{align}
Now, we consider the integral
\begin{align*}
\iint_{[2j-1, 2j]^2} ds_j d\widehat s_j \int_{D(\widehat y_j, 2\widehat \varepsilon_j)} \exp\bigg( -\frac{|\widehat x_j|^2}{16k}\bigg) d\widehat x_j
\end{align*}
as a sum of two integrals, for one which takes values when the variables $\widehat{s}_j$ and $s_j$ are restricted such that $|\widehat{s}_j-s_j|>\varepsilon^d$, and the other takes values when the variables $\widehat{s}_j$ and $s_j$ are restricted such that $|\widehat{s}_j-s_j| \leq \varepsilon^d$.
Then,
\begin{align*}
& \iint_{[2j-1, 2j]^2} ds_j d\widehat s_j \int_{D(\widehat y_j, 2\widehat \varepsilon_j)} \exp\bigg( -\frac{|\widehat x_j|^2}{16k}\bigg) d\widehat x_j\\
& = \iint\limits_{\substack{s_j, \widehat{s}_j \in [2j-1,2j],\\|\widehat{s}_j-s_j|>\varepsilon^d}}ds_jd\widehat{s}_j\int_{D(\widehat{y}_j,2\widehat{\varepsilon}_j)}\exp\bigg(-\frac{|\widehat{x}_j|^2}{16k}\bigg)d\widehat{x}_j + \iint\limits_{\substack{s_j, \widehat{s}_j \in [2j-1,2j],\\|\widehat{s}_j-s_j|\le\varepsilon^d}}ds_jd\widehat{s}_j\int_{D(\widehat{y}_j,2\widehat{\varepsilon}_j)}\exp\bigg(-\frac{|\widehat{x}_j|^2}{16k}\bigg)d\widehat{x}_j .
\end{align*}
For the first integral, we use the fact that $\exp(-|x|^2/16k) \le 1$ for all $x$, and the Lebesgue measure of the ball $D(\widehat{y}_j, 2\widehat{\varepsilon}_j)$ is $C {\widehat{\varepsilon_j}}^d = C\varepsilon^d/|\widehat{s}_j - s_j|$ (recall that $\widehat\varepsilon_j = \varepsilon/|\widehat{s}_j - s_j|^{H}$ and $Hd = 1$). For the second integral, we use $\int_{\mathbb{R}^d} \exp(-|x|^2/16k) dx < \infty$.
We get that
\begin{align*}
&\iint_{[2j-1, 2j]^2} ds_j d\widehat s_j \int_{D(\widehat y_j, 2\widehat \varepsilon_j)} \exp\bigg( -\frac{|\widehat x_j|^2}{16k}\bigg) d\widehat x_j\\
&\le \iint_{\substack{s_j, \widehat{s}_j \in [2j-1,2j],\\|\widehat{s}_j-s_j|>\varepsilon^d}} \frac{C\varepsilon^d}{|\widehat{s}_j-s_j|}ds_jd\widehat{s}_j + \iint_{\substack{s_j, \widehat{s}_j \in [2j-1,2j],\\|\widehat{s}_j-s_j|\leq\varepsilon^d}}C ds_jd\widehat{s}_j.
\end{align*}
Then by \eqref{int_est_log} and \eqref{Eq:log_lower_bd}, and also the assumption that $|z_1 - z_2| \le 4\varepsilon$, we have
\begin{align*}
&\iint_{[2j-1, 2j]^2} ds_j d\widehat s_j \int_{D(\widehat y_j, 2\widehat \varepsilon_j)} \exp\bigg( -\frac{|\widehat x_j|^2}{16k}\bigg) d\widehat x_j\\
& \le C\varepsilon^d\log\frac{1}{\varepsilon} + C\varepsilon^d
 \le C'\varepsilon^d\log\frac{1}{|z_1 - z_2|}.
\end{align*}
It follows that
\begin{align}\label{bound_F}
F \le C \iint_{\substack{z_1, z_2 \in E,\\|z_1-z_2|\le4\varepsilon}} \bigg(\log \frac{1}{|z_1-z_2|} \bigg)^k \mu(dz_1) \mu(dz_2).
\end{align}

Now consider $F + S$. Combining \eqref{bound_S} and \eqref{bound_F}, we have
\[ \mathbb{E}(I_\varepsilon^2)=F+S\le C\int_E \int_E \bigg(\log\frac{1}{|z_1-z_2|}\bigg)^k \mu(dz_1)\mu(dz_2) =: C_4. \]
By \eqref{assump}, $C_4$ is a finite constant.
Given the lower bound of $\mathbb{E}(I_\varepsilon)^2$ and the upper bound of $\mathbb{E}(I_\varepsilon^2)$, we deduce from \eqref{Eq:CS} that
\[ \mathbb{P}(I_\varepsilon>0) \ge C_3/C_4 > 0 \]
for all small $\varepsilon > 0$.
Because $C_3$ and $C_4$ are independent of $\varepsilon$, we can let $\varepsilon \rightarrow 0$ to conclude that the event that $E$ contains a $k$-tuple point $z = B_{s_1} =\cdots = B_{s_k}$ for some $s_1 \in [1, 2], s_2 \in [3, 4], \dots, s_k \in [2k-1, 2k]$ has positive probability. 
\end{proof}

\section{Multiple Points of Fractional Brownian Motion with $Hd > 1$}

The following result is the analogue of Theorem \ref{ktuplepointsHd=1} for $Hd > 1$.

\begin{thm}
\label{ktupleHd>1}
Let $B$ be a fractional Brownian motion in $\mathbb{R}^d$ with Hurst index $H$ such that $Hd>1$. If $E$ is a fixed, nonempty compact set in $\mathbb{R}^d$ with positive capacity with respect to the function $\phi(s) = s^{-k(d-1/H)}$, then $E$ contains $k$-tuple points with positive probability.
\end{thm}

\begin{proof}
It is easy to see that $E$ has positive $\phi$-capacity if and only if $\lambda E$ has positive $\phi$-capacity for all $\lambda > 0$, so we can assume that $E$ is contained in a ball of radius $1/3$ such that $|z_1-z_2| \le 2/3 < 1$.
We can find a probability measure $\mu$ on $E$ such that
\begin{align}\label{finite_energy2}
\int_E\int_E \frac{1}{|z_1-z_2|^{k(d-1/H)}} \mu(dz_1)\mu(dz_2) < \infty.
\end{align}
As before, for small $\varepsilon > 0$, we consider the integral
\begin{align*}
I_{\varepsilon}= \frac{1}{\varepsilon^{kd}}\int_{E}\mu(dz) \prod_{j=1}^k \int_{2j-1}^{2j}ds_j \textbf{K}_\varepsilon({|B_{s_j}-z|}).
\end{align*}
We prove a lower bound for $\mathbb{E}(I_\varepsilon)$ and an upper bound for $\mathbb{E}(I_\varepsilon^2)$.
The proof follows as in Theorem \ref{ktuplepointsHd=1} with changes after \eqref{integralSwithCOV}. 
With $Hd > 1$, \eqref{integralSwithCOV} now becomes 
\begin{align}\label{Eq:bound_S}
S \le C\iint_{\substack{z_1, z_2 \in E,\\|z_1-z_2|>4\varepsilon}}\mu(dz_1)\mu(dz_2) \prod_{j=1}^k \iint_{[2j-1, 2j]^2} \frac{1}{|\widehat{s}_j-s_j|^{Hd}}\exp\bigg(\frac{-|\widehat{y}_j|^2}{16k}\bigg) ds_j d\widehat{s}_j.
\end{align}
Splitting the integral in $ds_j d\widehat s_j$ into $L_j + M_j$ as before, we get that
\begin{align}\label{L+M}
\notag&\iint_{[2j-1, 2j]^2} \frac{1}{|\widehat{s}_j-s_j|^{Hd}}\exp\bigg(\frac{-|\widehat{y}_j|^2}{16k}\bigg) ds_j d\widehat{s}_j\\
\notag& = \iint_{\substack{s_j, \widehat s_j \in [2j-1, 2j],\quad\\|\widehat s_j - s_j| < |z_1 - z_2|^{1/H}}} \frac{ds_j d\widehat{s}_j }{|\widehat{s}_j-s_j|^{Hd}}\exp\bigg(\frac{-|\widehat{y}_j|^2}{16k}\bigg) + \iint_{\substack{s_j, \widehat s_j \in [2j-1, 2j],\quad\\|\widehat s_j - s_j| \ge |z_1 - z_2|^{1/H}}} \frac{ds_j d\widehat{s}_j }{|\widehat{s}_j-s_j|^{Hd}}\exp\bigg(\frac{-|\widehat{y}_j|^2}{16k}\bigg) \\
& \le \iint_{\substack{s_j, \widehat s_j \in [2j-1, 2j],\quad\\|\widehat s_j - s_j| < |z_1 - z_2|^{1/H}}} \frac{ds_j d\widehat{s}_j }{|z_1 - z_2|^{d}}|\widehat y_j|^d\exp\bigg(\frac{-|\widehat{y}_j|^2}{16k}\bigg) + \iint_{\substack{s_j, \widehat s_j \in [2j-1, 2j],\quad\\|\widehat s_j - s_j| \ge |z_1 - z_2|^{1/H}}} \frac{ds_j d\widehat{s}_j }{|\widehat{s}_j-s_j|^{Hd}}.
\end{align}
Since the function $y \mapsto |y|^d \exp(-|y|^2/16k)$ is bounded on $\mathbb{R}^d$, the first integral in \eqref{L+M} is 
\[ \le \frac{C}{|z_1-z_2|^{d-1/H}}. \]
Elementary calculations show that for any $a \in \mathbb{R}$, $0 < x < 1$ and $Hd > 1$,
\begin{align}\label{int_est_power}
\iint_{\substack{s, \widehat{s} \in [a, a+1],\\ |\widehat{s}-s| > x}} \frac{1}{|\widehat{s}-s|^{Hd}} ds \, d\widehat{s}
&= \frac{2(Hd-1)^{-1}}{x^{Hd-1}} + \frac{2(2-Hd)^{-1}}{x^{Hd-2}} + \frac{2}{(1-Hd)(2-Hd)},
\end{align}
which is bounded by $C/x^{Hd-1}$ for some constant $C$.
So the second integral in \eqref{L+M} is 
\[ \le \frac{C}{(|z_1-z_2|^{1/H})^{Hd-1}} = \frac{C}{|z_1-z_2|^{d-1/H}}. \]
It follows that
\[ \iint_{[2j-1, 2j]^2} \frac{1}{|\widehat{s}_j-s_j|^{Hd}}\exp\bigg(\frac{-|\widehat{y}_j|^2}{16k}\bigg) ds_j d\widehat{s}_j \le \frac{C}{|z_1-z_2|^{d-1/H}}, \]
and hence
\begin{align*}
S \le C \iint_{\substack{z_1, z_2 \in E,\\|z_1-z_2| > 4\varepsilon}} \frac{1}{|z_1-z_2|^{k(d-1/H)}} \mu(dz_1) \mu(dz_2).
\end{align*}

For $F$, we consider $|z_1 - z_2| \le 4\varepsilon$. The bound \eqref{Eq:bound_F} is still valid. For the integral
\begin{align*}
\iint_{[2j-1, 2j]^2} ds_j d\widehat s_j \int_{D(\widehat y_j, 2\widehat \varepsilon_j)} \exp\bigg( -\frac{|\widehat x_j|^2}{16k}\bigg) d\widehat x_j
\end{align*}
in \eqref{Eq:bound_F}, we consider it as the sum of two integrals over $|\widehat{s}_j - s_j| > \varepsilon^{1/H}$ and 
$|\widehat{s}_j - s_j| \le \varepsilon^{1/H}$:
\begin{align*}
& \iint_{[2j-1, 2j]^2} ds_j d\widehat s_j \int_{D(\widehat y_j, 2\widehat \varepsilon_j)} \exp\bigg( -\frac{|\widehat x_j|^2}{16k}\bigg) d\widehat x_j\\
& = \iint\limits_{\substack{s_j, \widehat s_j \in [2j-1, 2j],\\|\widehat s_j - s_j|> \varepsilon^{1/H}}} ds_j d\widehat s_j \int_{D(\widehat y_j, 2\widehat \varepsilon_j)} \exp\bigg( -\frac{|\widehat x_j|^2}{16k}\bigg) d\widehat x_j + \iint\limits_{\substack{s_j, \widehat s_j \in [2j-1, 2j],\\ |\widehat s_j - s_j| \le \varepsilon^{1/H}}} ds_j d\widehat s_j \int_{D(\widehat y_j, 2\widehat \varepsilon_j)} \exp\bigg( -\frac{|\widehat x_j|^2}{16k}\bigg) d\widehat x_j.
\end{align*}
For the first term, we bound the exponential function by 1 and integrate over the ball $D(\widehat{y}_j, 2\widehat{\varepsilon}_j)$. For the second term, we bound it by replacing the ball $D(\widehat{y}_j, 2\widehat{\varepsilon}_j)$ by all of $\mathbb{R}^d$ and note that $\int_{\mathbb{R}^d} \exp(-|\widehat{x}_j|^2/16k) d\widehat{x}_j$ is finite. It follows that
\begin{align*}
&\iint_{[2j-1, 2j]^2} ds_j d\widehat s_j \int_{D(\widehat y_j, 2\widehat \varepsilon_j)} \exp\bigg( -\frac{|\widehat x_j|^2}{16k}\bigg) d\widehat x_j\\
& \le \iint_{\substack{s_j, \widehat{s}_j \in [2j-1,2j],\\|\widehat{s}_j - s_j|> \varepsilon^{1/H}}}  \frac{C\varepsilon^d}{|\widehat{s}_j-s_j|^{Hd}}ds_j d\widehat{s}_j
+ \iint_{\substack{s_j, \widehat{s}_j \in [2j-1,2j],\\|\widehat{s}_j - s_j|\le \varepsilon^{1/H}}}C ds_1 d\widehat{s}_j.
\end{align*}
Then by \eqref{int_est_power} and $|z_1-z_2| \le 4\varepsilon$, 
\begin{align*}
&\iint_{[2j-1, 2j]^2} ds_j d\widehat s_j \int_{D(\widehat y_j, 2\widehat \varepsilon_j)} \exp\bigg( -\frac{|\widehat x_j|^2}{16k}\bigg) d\widehat x_j\\
& \le  \frac{C\varepsilon^d}{\varepsilon^{d-1/H}} + C\varepsilon^{1/H}\\
& \le \frac{C'\varepsilon^d}{|z_1-z_2|^{d-1/H}}.
\end{align*}
Hence
\begin{align*}
F \le C \iint_{\substack{z_1, z_2 \in E,\\|z_1-z_2| \le 4\varepsilon}} \frac{1}{|z_1-z_2|^{k(d-1/H)}} \mu(dz_1)\mu(dz_2).
\end{align*}

Therefore, we obtain the following bound for $\mathbb{E}(I_\varepsilon^2)$:
\[ \mathbb{E}(I_\varepsilon^2) = S + F \le C \int_E \int_E \frac{1}{|z_1-z_2|^{k(d-1/H)}} \mu(dz_1)\mu(dz_2). \]
This bound is finite by \eqref{finite_energy2} and is independent of $\varepsilon$. The rest of the proof is the same as in Theorem \ref{ktuplepointsHd=1}.
\end{proof}

%\begin{remark}\label{Rmk:Hd<1}
%The proof also indicates that if $Hd<1$ and $E$ is any fixed, nonempty set in $\mathbb{R}^d$, then $E$ contains $k$-tuple points with positive probability. Indeed, if we take $\mu$ to be the point mass at a point in $E$ and follow the proof, then $S$ now becomes 0 and for $F$, \eqref{int_est_power} is now bounded by a constant and the integral in \eqref{F_integral} is bounded by $C \varepsilon^d$, so $\mathbb{E}(I_\varepsilon^2)$ is again bounded above by a constant. 
%\end{remark}

\section{Acknowledgements}
We would like to thank Dr.~Yimin Xiao for his guidance in the direction of our research project. 
This paper was written as a part of the SURIEM REU at Michigan State University, and we would finally like to thank Dr. Robert Bell for his contributions to organizing this program. The SURIEM REU is supported by MSU, the NSA, and the NSF.

\end{document}